\documentclass{article}
\usepackage{amsfonts,amsmath,amsthm,amscd}
\usepackage[all]{xy}
\newtheorem{thm}{Theorem}[section]
\newtheorem{prop}[thm]{Proposition}
\newtheorem{lem}[thm]{Lemma}
\newtheorem{cor}[thm]{Corollary}
\theoremstyle{definition}
\newtheorem{claim}[thm]{Claim}
\newtheorem{defn}[thm]{Definition}

\newtheorem{rem}[thm]{Remark}

\makeatletter
\@addtoreset{equation}{section}

\makeatother

\newcommand{\mf}{\mathcal F}
\newcommand{\mm}{\mathcal M}

\newcommand{\mo}{\mathcal O}
\newcommand{\mv}{\mathcal V}
\newcommand{\mx}{\mathcal X}
\newcommand{\my}{\mathcal Y}
\newcommand{\mbar}{\overline{\mm}}
\newcommand{\mgp}{\overline{\mm}^{\text{\rm gp}}}

\newcommand{\Fbar}{\overline F}
\newcommand{\Kbar}{\overline K}

\newcommand{\ql}{\mathbb Q_\ell}
\newcommand{\zl}{\mathbb Z_\ell}

\newcommand{\flbar}{\overline{\mathbb F}_\ell}
\newcommand{\zz}{\mathbb Z}
\newcommand{\nn}{\mathbb N}
\newcommand{\qq}{\mathbb Q}

\newcommand{\tr}{\mathop{\mathrm{Tr}}\nolimits}
\newcommand{\spec}{\mathop{\mathrm{Spec}}\nolimits}
\newcommand{\Aut}{\mathop{\mathrm{Aut}}\nolimits}
\newcommand{\gal}{\mathop{\mathrm{Gal}}\nolimits}

\title{$\ell$-independence of the trace of local monodromy in a relative case}
\author{Hiroki Kato
\thanks{Graduated School of Mathematical Science, The University of Tokyo. 
\texttt{Email: hiroki@ms.u-tokyo.ac.jp}}
\\with an appendix by 
Qing Lu and 
Weizhe Zheng}
\date{}

\begin{document}
\maketitle
\begin{abstract}
For a family of varieties, 
we prove that 
the alternating sum of the traces of ``local" monodromy 
acting on the $\ell$-adic \'etale cohomology groups of the generic fiber 
is an integer which is independent of $\ell$. 
\end{abstract}

\section{Introduction}
Ochiai proved an $\ell$-independence result for a variety over a local field. 
To be more precise, 
let $K$ be a henselian discrete valuation field, 
$\overline K$ a separable closure of $K$, 
and $X$ a variety over $K$. 
He proved that 
for an element $\sigma$ of the inertia subgroup $I_K$ of the absolute Galois group of $K$, 
the alternating sum 
\begin{equation*}
\sum_q(-1)^q\tr(\sigma,H_c^q(X_{\overline K},\ql))
\end{equation*}
is an integer independent of $\ell$ distinct from the residual characteristic \cite[Theorem B]{O}. 

By the same method, Vidal established an equivariant version of his result, 
that is, for a variety $X$ over $K$ with an action of a finite group $G$ 
and for an element $(g,\sigma)\in G\times I_K$, the alternating sum 
\begin{equation*}
\sum_q(-1)^q\tr((g,\sigma),H_c^q(X_{\overline K},\ql))
\end{equation*}
is an integer independent of $\ell$ distinct from the residual characteristic \cite[Proposition 4.2]{V1}. 
She applied it to compare wild ramification of the elements $[R\Gamma_c(Z_{\Kbar},\mf_i)]$ 
in the Grothendieck group of the category of Galois representations 
for constructible \'etale $\flbar$-sheaves $\mf_i$ $(i=1,2)$ on a variety $Z$ over a henselian discrete valuation field $K$ \cite[Th\'eor\`eme 3.1]{V1}. 


She also worked in a relative situation. 
Let $f:Z\to Y$ be a morphism of varieties over a henselian discrete valuation field $K$ with excellent integer ring. 
In \cite{V}, to compare wild ramification of the elements $[Rf_!\mf_i]$ 
in the Grothendieck group of the category of constructible \'etale $\flbar$-sheaves on $Y$ 
for constructible \'etale $\flbar$-sheaves $\mf_i$ $(i=1,2)$ on $Z$, 
she proved the following $\ell$-independence result for relative curves \cite[Proposition 2.2.1]{V}. 
Since it concerns on monodromy action on the cohomology of the generic fiber of a relative curve, 
we start with a curve $X$ over a field $L$ finitely generated over $K$, 
although she did with a relative curve. 
We consider an action of a finite group $G$ on $X$ over $L$. 
Let $\overline L$ be a separable closure of $L$. 
Then, we have a monodromy action composed with the action of $G$, 
that is an action of $G\times\gal(\overline L/L)$ 
on the cohomology groups $H_c^q(X_{\overline L},\ql)$. 
She proved that 
for every $g\in G$ and $\sigma\in\gal(\overline L/L)$ in Vidal's ramified part, i.e., an element coming from the inertia group of a valuation ring over $S=\spec\mo_K$ (see Section \ref{ramified part} for the precise definition), 
the alternating sum 
$$
\sum_q(-1)^q\tr((g,\sigma),H_c^q(X_{\overline L},\ql))
$$
is an integer independent of $\ell$ distinct from the residual characteristic of $K$. 

We generalize her $\ell$-independence result to the case of a general family. 
We work over an arbitrary excellent noetherian scheme of dimension $\leq2$. 
\begin{thm}\label{main}
Let $S$ be an excellent noetherian scheme of dimension $\leq2$. 
Let $L$ be a field and $\spec L\to S$ be a morphism of schemes such that $L$ is finitely generated over the residue field at the image. 
Let $X$ be a scheme separated and of finite type over $L$ on which a finite group $G$ acts admissibly. 
Let $\overline L$ be a separable closure of $L$. 
Then, for every $g\in G$ and 
$\sigma\in\gal(\overline L/L)$ in Vidal's ramified part, 
i.e., an element coming from the inertia group of a valuation ring over $S$, 
the alternating sum 
$$
\sum_q(-1)^q\tr((g,\sigma),H_c^q(X_{\overline L},\ql))
$$
is an integer independent of a prime number $\ell$ invertible on $S$. 
\end{thm}

The assumption on dimension of $S$ will be used to resolve singularities (Corollary \ref{Lipman}). 

We make comments on Vidal's ramified part. 
Vidal's ramified part is a higher dimensional analogue of the inertia subgroup for a henselian discrete valuation field. 
It is defined for a morphism $\spec L\to S$ of schemes with $L$ a field 
and is a subset (not a subgroup in general) of $\gal(\overline L/L)$, 
where $\overline L$ is a separable closure of $L$ (Definition \ref{Vidal}). 
For example, if $S$ is $\spec\mo_K$ for a henselian discrete valuation ring $\mo_K$ 
and $\spec L\to S$ is the natural open immersion $\spec K\to \spec\mo_K$, 
then Vidal's ramified part for $\spec L\to S$ coincides with the inertia subgroup of $K$. 
More generally, if $S$ is noetherian 
and if $L$ is the function field of a normal connected scheme $\my$ proper over $S$, 
then every element of Vidal's ramified part belongs to 
the inertia subgroup at some geometric point $y$ of $\my$, 
i.e., the image of the morphism $\pi_1(\spec L\times_\my\my_{(y)},\alpha)\to\gal(\overline L/L)$ for some geometric point $\alpha$ lying above $\spec \overline L$ (Lemma \ref{char of E}). 
In other words, elements of Vidal's ramified part are ramified along each ``compactification'' of $\spec L$ over $S$. 


We explain the strategy of the proof of Theorem \ref{main}. First, we briefly recall the proof of Ochiai's $\ell$-independence result for a variety over a henselian discrete valuation field. 
Ochiai reduced the proof to the semi-stable reduction case with a finite group action by taking an alteration using a result of de Jong. 
Then, he used the weight spectral sequence by Rapoport-Zink to describe the Galois action on the \'etale cohomology groups in terms of geometry of the closed fiber and to deduce $\ell$-independence from the Lefschetz trace formula. 

To work in a relative setting, Vidal modified Ochiai's proof using log structures of Fontaine-Illusie instead of the weight spectral sequence. 
Our proof is based on her idea. 
We also take an alteration using a refinement of de Jong's result due to Gabber (Lemma \ref{alt}) 
to reduce the problem to the ``semi-stable reduction'' case with a finite group action (Proposition \ref{localmonod}). 
Here, under the above notations, reduction is considered along a ``compactification'' of $\spec L$ over $S$ and ``semi-stable'' means that the variety $X$ is the generic fiber of a ``nodal fibration'' (``pluri nodal fibration'' in \cite{dJ2}, see Section \ref{nodal} for the precise definition) over a ``compactification'' of $\spec L$ over $S$. 
Then, we apply Nakayama's log smooth and proper base change theorem to the ``nodal fibration'', which we will denote by $\mx\to\my$, with the natural log structures. 
It gives us a canonical isomorphism between the \'etale cohomology group of the generic fiber and the Kummer \'etale cohomology group of the log geometric fiber $\mx_{\tilde y}$ over a closed point $y$ of $\my$. 
To describe the Kummer \'etale cohomology in terms of the usual \'etale cohomology, we compute the ``log nearby cycle complex'' $R\varepsilon_*\ql$, where $\varepsilon$ is the morphism forgetting the log structure of $\mx_{\tilde y}$. 
Then, we can deduce the $\ell$-independence from the Lefschetz trace formula. 

Theorem \ref{main} can be used to remove a reduction argument 
in the proof of the main result of \cite{V} on comparison of wild ramification, 
in which she decomposes a morphism into relative curves. 
Further, in a forthcoming paper, we plan to apply Theorem \ref{main} to compare wild ramification of nearby cycle complexes. 


Recently Q.\ Lu and W.\ Zheng also obtained an $\ell$-independence result in equal characteristic case. 
They proved that, for arbitrary $S$ of characteristic $p>0$, 
the trace of each $\sigma\in\gal(\overline L/L)$ in Vidal's ramified part 
acting on the cohomology group of each degree 
is independent of $\ell\neq p$ \cite[Theorem 1.4 and Remark 2.18(3)]{LZ}. 

After the first version of this article was circulated, 
Q.\ Lu and W.\ Zheng informed us 
that Theorem \ref{main} can be deduced from results 
\cite[Corollary 1.3]{LZ} and \cite[Proposition 4.15]{Z} on compatibility of $\ell$-adic sheaves 
and that the assumptions on $S$ and $\spec L\to S$ can be removed, 
i.e., Theorem \ref{main} holds for an arbitrary scheme $S$ 
and an arbitrary morphism $\spec L\to S$ with $L$ being a field. 
We include their comment as an appendix. 

This paper is organized as follows. 
In Section \ref{ramified part} we give a definition of Vidal's ramified part in our setting and show some basic properties. 
We recall Gabber's refinement (Lemma \ref{alt}) of de Jong's result in Section \ref{nodal} 
and basic properties of log structures in Section \ref{log}. 
After formulating ``monodromy action composed with a finite group action'' in an equivariant setting in Section \ref{G-sheaves}, 
we prove the main result in Section \ref{indep}. 

\paragraph{Acknowledgment}
The author would like to thank my advisor T.\ Saito 
for helpful discussions, 
reading this article carefully, 
and giving comments which improved various points of the article. 
He also thanks 
D.\ Takeuchi for helpful discussions 
and Q.\ Lu and W.\ Zheng for sharing another proof of Theorem \ref{main} and kindly suggesting the author to include it in the appendix. 
The research was supported by the Program for Leading Graduate Schools, MEXT, Japan and also by JSPS KAKENHI Grant Number 18J12981.

\section{Vidal's ramified part}\label{ramified part}

\begin{defn}\label{Vidal}
Let $S$ be a scheme, $K$ be a field, and $\spec K\to S$ a morphism of schemes. 
We take a separable closure $\Kbar$ of $K$. 
We define a subset $E_{K/S}$ of $\gal(\Kbar/K)$, 
which we call Vidal's ramified part, 
as follows: 
We consider a following commutative diagram: 
\begin{eqnarray}\label{diagram}
\begin{xy}
\xymatrix{
\spec\Fbar\ar[r]^{\bar\iota}\ar[d]&\spec\Kbar\ar[d]\\
\spec F\ar[r]^{\iota}\ar[d]&\spec K\ar[d]\\
\spec\mo_F\ar[r]&S,}
\end{xy}
\end{eqnarray}
where $\mo_F$ is a strictly henselian valuation ring, $F$ its field of fraction, and $\Fbar$ a separable closure of $F$. 
We define a subset $E_{(\iota,\bar\iota)}$ of $\gal(\Kbar/K)$ as the image of the natural map 
$\gal(\Fbar/F)\to\gal(\Kbar/K)$. 
We define $E_{K/S}$ to be the closure of the union of $E_{(\iota,\bar\iota)}$ for all diagrams as above. 
\end{defn}

We study functorial properties of Vidal's ramified part: 
\begin{lem}\label{on E}
Let $S$ be a scheme. 
\begin{enumerate}
\item\label{functoriality}
$($c.f. $\text{\rm\cite[Proposition 2.1.1]{V1}}$, $\text{\rm\cite[Lemma 2.4]{KH}})$ 
We consider a commutative diagram 
$$
\begin{xy}
\xymatrix{\spec\overline{K'}\ar[r]\ar[d]&\spec \Kbar\ar[d]\\\spec K'\ar[r]\ar[d]&\spec K\ar[d]\\S'\ar[r]&S}
\end{xy}
$$
of schemes such that $K$ and $K'$ are fields 
and that $\Kbar$ and $\overline{K'}$ are separable closures of $K$ and $K'$ respectively.  
Then, the natural morphism 
$\gal(\overline{K'}/K')\to\gal(\Kbar/K)$ 
induces a map 
$E_{K'/S'}\to E_{K/S}$. 
\item\label{ram part pinsep}
Let $K$ be a field and $\spec K\to S$ be a morphism of schemes. 
Let $L$ be a finite purely inseparable extension of $K$. 
We take a separable closure $\overline L$ (resp. $\Kbar$) of $L$ (resp. $K$) and fix an embedding $\Kbar\to\overline L$ over $K$. 
Then, the map 
$E_{L/S}\to E_{K/S}$ 
defined in \ref{functoriality} is bijective. 
\end{enumerate}
\end{lem}
\begin{proof}
\begin{enumerate}
\item
Clear from the definition. 
\item
Since the injectivity of $E_{L/S}\to E_{K/S}$ follows from the bijectivity of the map between the absolute Galois groups, it suffices to prove that the map $E_{L/S}\to E_{K/S}$ is surjective. 
Let 
$$
\begin{xy}
\xymatrix{
\spec\Fbar\ar[r]\ar[d]&\spec\overline L\ar[d]\\
\spec F\ar[r]\ar[d]&\spec K\ar[d]\\
\spec\mo_F\ar[r]&S}
\end{xy}
$$
be a diagram such that 
$\mo_F$ is a strictly henselian valuation ring, $F$ its field of fractions, and $\Fbar$ an algebraic closure of $F$. 
Let $E$ be the minimum subfield of $\Fbar$ containing $F$ and $L$. 
Then, $E$ is finite and purely inseparable over $F$. 
Further, the normalization of $\mo_F$ in $E$ is a strictly henselian valuation ring. 
Thus, the map in the assertion is surjective. 
\end{enumerate}
\end{proof}

Let $\spec K\to S$ be as above and assume it is ``essentially of finite type'' (see Definition \ref{cptfn} below). 
In Lemma \ref{char of E} below, we characterize Vidal's ramified part $E_{K/S}$ 
as the subset of $\gal(\Kbar/K)$ consisting of elements which have a fixed geometric point for every compactification over $S$ and every finite Galois extension of $K$. 
To be more precise, we make the following definition:

\begin{defn}\label{cptfn}
Let $S$ be a scheme, $K$ a field, and $\spec K\to S$ a morphism of schemes. 
\begin{enumerate}
\item
\label{fg}
We say that $\spec K\to S$ is {\it essentially of finite type} if 
the field $K$ is a finitely generated extension of the residue field at the image. 
\item
Assume that $\spec K\to S$ is essentially of finite type. 
A {\it compactification of} $K$ {\it over} $S$ is an integral scheme $\my$ proper over $S$ with an $S$-morphism $\spec K\to\my$ which induces an isomorphism between $\spec K$ and the generic point of $\my$. 
\end{enumerate}
\end{defn}
Logically, in the proof of the main theorem, we use only the implication $1\Rightarrow 2$ in Lemma \ref{char of E} below and do not use the other implication $2\Rightarrow 1$. But, this characterization shows that Vidal's ramified part $E_{K/S}$ is the largest subset to which we can apply the argument in Section \ref{indep}. 
\begin{lem}[{cf. \cite[6.1, 6.3]{V}}]\label{char of E}
Let $S$ be a noetherian scheme 
and $K$ be a field with a morphism $\spec K\to S$ which is essentially of finite type. 
We take a separable closure $\Kbar$ of $K$. 
For $\sigma\in\gal(\Kbar/K)$, the following are equivalent. 
\begin{enumerate}
\item
$\sigma\in E_{K/S}$. 
\item
For every normal compactification $\my$ of $K$ over $S$ 
and every finite Galois extension $L$ of $K$ contained in $\Kbar$, 
there exists a geometric point $\bar v$ of the normalization $\mv$ of $\my$ in $L$ such that $\sigma \bar v=\bar v$ under the natural action of $\gal(\Kbar/K)$ on $\mv$. 
\end{enumerate}
\end{lem}
The proof is essentially the same as that in \cite[6.1]{V}. But we include the proof for the completeness. 
\begin{proof}
$1\Rightarrow 2$. 
Let $\sigma\in E_{K/S}$ and let $\overline\sigma\in\gal(L/K)$ be the image of $\sigma$ by $\gal(\Kbar/K)\to\gal(L/K)$. Then, $\overline\sigma$ is in the image of 
$\gal(\Fbar/F)\to\gal(\Kbar/K)\to\gal(L/K)$ 
for some diagram (\ref{diagram}). 
Let $E$ be the minimum subfield of $\Fbar$ containing $L$ and $F$ 
and let $\mo_E$ be the normalization of $\mo_F$ in $E$. 
Then, since $\mo_F$ is a henselian valuation ring, $\mo_E$ is also a henselian valuation ring by \cite[Chapitre VI \S8.6 Proposition 6]{B}. 
Then, by the valuative criterion of proper morphisms \cite[Th\'eor\`eme 7.3.8]{ega2}, there exists a unique $S$-morphism $\spec\mo_E\to\mv$ extending the composite $\spec E\to\spec L\to\mv$. 
Then, since $\mo_F$ is strictly henselian, the geometric point $\bar v$ of $\mv$ defined by the closed point of $\spec\mo_E$ is fixed by $\sigma$. 

$2\Rightarrow 1$. Let $\sigma\in\gal(\Kbar/K)$ be an element satisfying the condition 2. To show $\sigma\in E_{K/S}$, it suffices to show that for every finite Galois extension $L$ of $K$ contained in $\Kbar$, the image of $\sigma$ in $\gal(L/K)$ is in the image of $\gal(\Fbar/F)\to\gal(\Kbar/K)\to\gal(L/K)$ 
for some diagram (\ref{diagram}). 

We take an inverse system $(\my_i)_{i\in I}$, indexed by a directed set $I$, of normal compactifications of $K$ over $S$ which is cofinal in the category of normal compactifications of $K$ over $S$. 
We denote the normalization of $\my_i$ in $L$ by $\mv_i$. 
\begin{claim}
There exists a compatible system $(\bar v_i\to\mv_i)_{i\in I}$ of geometric points fixed by $\sigma$, 
that is, each $\bar v_i$ is a geometric point of $\mv_i$ fixed by $\sigma$ 
and if there is a morphism $\my_i\to\my_j$ of compactifications of $\spec K$ over $S$, 
then $\bar v_j$ is the image of $\bar v_i$ by the induced morphism $\mv_i\to\mv_j$. 
\end{claim}
\begin{proof}
Let $\gamma_\sigma:\mv_i\to\mv_i\times_{\my_i}\mv_i$ be the graph of $\sigma:\mv_i\to\mv_i$. 
We define, for each $i\in I$, 
a closed subscheme $\mv_i^\sigma$ of $\mv_i$ 
to be the pullback $\gamma_\sigma^*(\Delta_{\mv_i})$, where $\Delta_{\mv_i}$ is the diagonal in $\mv_i\times_{\my_i}\mv_i$. 
Then, a point $v\in\mv_i$ is in $\mv_i^\sigma$ if and only if a geometric point over $v$ is fixed by $\sigma$. 
Here, $\mv_i^\sigma$ is nonempty by the condition 2, 
and is quasi-compact. 
Thus, the inverse limit $\varprojlim_{i\in I}\mv_i^\sigma$ in the category of topological spaces is nonempty. 
We take an element $(v_i)_{i\in I}$ of the inverse limit $\varprojlim_{i\in I}\mv_i^\sigma$ and a separable closure $\Omega$ of the inductive limit $\varinjlim_{i\in I}\kappa(v_i)$ of the residue fields $\kappa(v_i)$. 
Then $(\spec\Omega\to\mv_i)_{i\in I}$ is a compatible system of fixed geometric points. 
\end{proof}
We denote the image of $\bar v_i$ in $\my_i$ by $\bar y_i$. 
Let $\widetilde\my$ (resp. $\widetilde\mv$) be the inverse limit $\varprojlim_{i\in I}\my_{i(\bar y_i)}$ (resp. $\varprojlim_{i\in I}\mv_{i(\bar v_i)}$) 
and let $\widetilde K$ (resp. $\widetilde L$) be the fraction field of $\widetilde\my$ (resp. $\widetilde\mv$). 
Since $\sigma$ fixes every $\bar v_i$, its image is in the image of the map $\gal(\widetilde L/\widetilde K)\to\gal(L/K)$.
Here we note that $\widetilde\my$ is the spectrum of a strictly henselian valuation ring. 
In fact, if we write $\widetilde\my=\spec\widetilde A$, 
then the natural morphism  
$\widetilde A
\to(\varinjlim_{i\in I}\mo_{\my_i,y_i})^{\mathrm{sh}}$
is an isomorphism, 
where $y_i$ is the point of $\my_i$ lying under $\bar y_i$. 
Since $\varinjlim_{i\in I}\mo_{\my_i,y_i}$ is a valuation ring by Lemma \ref{val ring} below, $\widetilde A$ is also a valuation ring by Lemma \ref{sh of val ring} below. 
Thus, the assertion follows. 
\end{proof}

\begin{lem}\label{val ring}
Let $S$ be a noetherian scheme 
and $\spec K\to S$ be a morphism which is essentially of finite type. 
We take an inverse system $(\my_i)_{i\in I}$, indexed by a directed set $I$, of normal compactifications of $K$ over $S$ which is cofinal in the category of normal compactifications of $K$ over $S$. 
Then, for every element $(y_i)_{i\in I}$ of the set $\varprojlim_{i\in I}\my_i$, the ring $\varprojlim_{i\in I}\mo_{\my_i,y_i}$ is a valuation ring. 
\end{lem}
This seems to be well known, but the author could not find a suitable reference. So we include a proof. 

Note that the assumption that $S$ is noetherian assures the existence of a compactification, by Nagata's compactification theorem. 
\begin{proof}
We put $O=\varprojlim_{i\in I}\mo_{\my_i,y_i}$
We check that for any $x\in K$, we have either $x\in O$ or $x^{-1}\in O$. 
We fix an index $i_0\in I$. 
We may assume that $\my_{i_0}$ is affine; $\my_{i_0}=\spec A$. 
Then, any $x\in K$ can be written as $x=f/g$ for some $f,g\neq0\in A$. 
We take an index $i\in I$ such that $\my_i$ dominates the blow up of $\my_{i_0}$ along the ideal $I$ of $A$ generated by $f$ and $g$. Then, either $f/g$ or $g/f$ belongs to 
$\mo_{\my_i,y_i}$, and hence, the assertion follows. 
\end{proof}

\begin{lem}\label{sh of val ring}
Let $A$ be a valuation ring with residue field $k$. 
Let $k^{\mathrm{sep}}$ be a separable closure of $k$ 
and $A^{\mathrm{sh}}$ be the strict henselization of $A$ with respect to $k\to k^{\mathrm{sep}}$. 
Then, $A^{\mathrm{sh}}$ is also a valuation ring. 
\end{lem}
\begin{proof}
By the definition of strict henselization \cite[D\'efinition 16.8.7]{ega4}, $A^{\mathrm{sh}}=\varinjlim A'$, where the limit is taken over all essentially \'etale local $A$-algebras $A'$ with an $A$-morphism $A'\to k^{\mathrm{sep}}$. 
Thus, it suffices to show that every essentially \'etale local $A$-algebra $A'$ is also a valuation ring. 
Using Zariski's main theorem, we can take a finite $A$-algebra $B$ which is normal, a maximal ideal $\mathfrak m$ of $B$ at which $A\to B$ is \'etale, and an $A$-isomorphism $B_\mathfrak m\cong A'$. 
By \cite[Chapitre VI \S8.6 Proposition 6]{B}, $B_\mathfrak m$ is a valuation ring. 
\end{proof}

Finally, we mention that Vidal's ramified part defined here recovers Vidal's original one. 
We assume that $S$ is an excellent trait. 
Let $Y$ be an $S$-scheme separated of finite type which is normal and connected. 
Let $K$ be the function field of $Y$. 
We take a separable closure $\Kbar$ of $K$. 
Let $E_{K/S}(Y)$ be the subset of $\pi_1(Y,\spec\Kbar)$ which Vidal calls ``{\it la partie g\'en\'eriquement ramifi\'ee}'' in \cite[1.2]{V}. 
By Gabber's valuative criterion \cite[6.3]{V}, we see the following.  
\begin{lem}\label{Vidal original}
The natural surjection $\gal(\Kbar/K)\to\pi_1(Y,\spec\Kbar)$ induces a surjection $E_{K/S}\to E_{K/S}(Y)$.  
\end{lem}

\begin{rem}
In \cite[1.2]{V}, a subset $E_{Y/S}$ of $\pi_1(Y,\spec\Kbar)$, which she calls ``{\it la partie ramifi\'ee}'', is defined. 
This subset contains the subset $E_{K/S}(Y)$. By Gabber's valuative criterion \cite[6.1]{V}, it coincides with the subset defined by replacing $\spec K$ by $Y$ in Definition \ref{Vidal}. 
In the case $Y$ is regular, we have $E_{Y/S}=E_{K/S}(Y)$ \cite[6.4]{V}, but we do not know whether this equality holds or not in the general case, i.e., for a normal $Y$.  
\end{rem}

\section{Alterations and nodal fibrations}\label{nodal}
We introduce some terminologies used in the proof of the main theorem. 

\begin{defn}[{\cite[7.1]{dJ}}]
Let $X$ be a regular scheme with an action of a finite group $G$ 
and $D$ a divisor of $X$ with simple normal crossings which is $G$-stable. 
We say that $D$ is a divisor with {\it $G$-strict normal crossings} 
if the $G$-orbit of every irreducible component of $D$ is a disjoint union of irreducible components of $D$. 
\end{defn}

\begin{defn}
Let 
$f:\mathcal X\to\mathcal Y$ be a 
morphism of schemes which is flat and of finite presentation. 
\begin{enumerate}
\item \cite[2.21]{dJ} We say $f$ is a {\it nodal curve} if every geometric fiber of $f$ is a connected curve whose singularities are at most  ordinary double points. 
\item \cite[4.4.1]{V1} (cf. 
\cite[2.22]{dJ})
Let $G$ be a finite group acting on $\mx$ and $\my$ so that $f:\mathcal X\to\mathcal Y$ is $G$-equivariant. 
We say $f$ is 
a $G$-{\it split nodal curve} 
if $f$ is a nodal curve and if for every point $y=\spec\kappa(y)$ of $\my$, 
\begin{enumerate}
\item every singular point of $\mx\times_\my y$ is rational over $\kappa(y)$, 
\item every irreducible component of $\mx\times_\my y$ is smooth over $\kappa(y)$, 
\item the $G$-orbit of any irreducible component of $\mx\times_\my y$ is a disjoint union of irreducible components of fibers of $f:\mx\to\my$ (not necessarily the fiber $\mx\times_\my y$). 
\end{enumerate}
\end{enumerate}
\end{defn}
Note that any base change of a nodal curve is also a nodal curve and that any base change of a $G$-split nodal curve by any $G$-equivariant morphism is also a $G$-split nodal curve.

\begin{defn}\label{datum}
Let $G$ be a finite group. 
\begin{enumerate}
\item
A {\it $G$-split nodal fibration datum} is a following datum: 
\begin{enumerate}
\item a sequence of $G$-split nodal curves $f_i:\mathcal X_i\to\mathcal X_{i-1}$ ($i=1,\ldots,d$); 
$$
\mathcal X_d
\to
\mathcal X_{d-1}
\to
\cdots
\to
\mathcal X_1
\to\mathcal X_0,
$$
\item a $G$-stable proper closed subset $Z_0$ of $\mathcal X_0$, 
\item finitely many disjoint sections $\sigma_{ij}:\mathcal X_{i-1}\to\mathcal X_i$ ($i=1,\ldots,d$) into the smooth locus of $f_i:\mathcal X_i\to\mathcal X_{i-1}$ 
which are permuted by $G$. 
\end{enumerate}
satisfying the following condition; 
if we define for $i=1,\ldots,d$ a closed subset $Z_i$ of $\mathcal X_i$ inductively by 
$$
Z_i=f_i^{-1}(Z_{i-1})\cup\bigcup_j\sigma_{ij}(\mathcal X_{i-1}),
$$ 
the nodal curve $f_i:\mathcal X_{i}\to\mathcal X_{i-1}$ is smooth over $\mathcal X_{i-1}\setminus Z_{i-1}$ for every $i$. 
\item
We say a $G$-split nodal fibration datum 
$(\mathcal X_d\to\mathcal X_{d-1}\to\cdots\to\mathcal X_1\to\mathcal X_0,Z_0,\{\sigma_{ij}\})$ 
is {\it strictly $G$-split} if 
$\mx_i$ is regular and 
the closed subscheme $Z_i$ defined as above is a divisor with $G$-strict normal crossings 
for every $i=0,\ldots,d$. 
\end{enumerate}
\end{defn}

\begin{defn}\label{fib}
A {\it $G$-split nodal fibration} (resp. {\it strictly $G$-split nodal fibration}) is a $G$-equivariant morphism $f:\mx\to\my$ of excellent schemes
such that there exists a $G$-split nodal fibration datum (resp. strictly $G$-split nodal fibration datum) 
$$
(
\mx=
\mathcal X_d
\to
\mathcal X_{d-1}
\to
\cdots
\to
\mathcal X_1
\to\mathcal X_0
=\my,
Z_0,
\{\sigma_{ij}\}_{i,j}
)
$$
such that 
the composite 
$
\mx=
\mathcal X_d
\to
\mathcal X_{d-1}
\to
\cdots
\to
\mathcal X_1
\to\mathcal X_0
=\my
$
coincides with $f$.
\end{defn}
We note that any base change of a $G$-split nodal fibration by any $G$-equivariant morphism is also a $G$-split nodal fibration. 

When the group $G$ is trivial, we omit $G$ from these terminologies; 
for example, a $G$-split nodal fibration for trivial $G$ will be simply called {\it a split nodal fibration}. 


%

Recall that an alteration is a proper generically finite surjection of integral noetherian schemes \cite[2.20]{dJ}. 
We use a following refinement Lemma \ref{alt} of de\ Jong's theorem \cite[Theorem 5.9]{dJ2} due to Gabber, 
to reduce the proof of our main theorem to the case where we can take a nodal fibration as an ``integral model''. 
\begin{defn}[{\cite[5.3]{dJ2}}]
Let $\mx$ (resp.\ $\mx'$) be an integral noetherian scheme with an action of a finite group $G$ (resp.\ $G'$). 
A {\it Galois alteration} $(\phi,\alpha):(\mx',G')\to(\mx,G)$ is a pair of 
a surjective homomorphism $\alpha:G'\to G$ of finite groups 
and an alteration $\phi:\mx'\to\mx$ which is $G'$-equivariant via $\alpha$ 
such that, 
writing $\Gamma$ for the kernel of the natural map $G'\to\Aut(X)$, the fixed part $K(\mx')^\Gamma$ of the function field $K(\mx')$ of $\mx'$ is purely inseparable over the function field $K(\mx)$ of $\mx$. 
\end{defn}

\begin{lem}[{\text{\rm\cite[Lemme 4.4.3]{V1}}}]
\label{alt}
Let $\mx$ and $\my$ be integral excellent noetherian schemes with an action of a finite group $G$ 
and $f:\mx\to\my$ a $G$-equivariant morphism separated and of finite type. 
Let $Z$ be a proper closed subset of $\mx$. 
Assume that the geometric generic fiber of $f$ is irreducible of dimension $d$. 
Then, there exist 
\begin{itemize}
\item a surjective homomorphism $\alpha:G'\to G$ of finite groups
\item projective Galois alterations $(\psi,\alpha):(\mathcal Y',G')\to(\mathcal Y,G)$ and $(\phi,\alpha):(\mathcal X',G')\to(\mx,G)$, 
\item a $G'$-split nodal fibration datum $($Definition \ref{datum}$)$
$$
(
\mathcal X_d
\to
\mathcal X_{d-1}
\to
\cdots
\to
\mathcal X_1
\to\mathcal X_0,
Z_0,\{\sigma_{ij}\})
$$
with $\mathcal X_d=\mathcal X'$ and $\mathcal X_0=\mathcal Y'$, 
\end{itemize}
such that 
$Z_d$ defined as in Definition \ref{datum} contains the pullback $\phi^{-1}(Z)$ 
and that 
the following diagram commutes: 
$$
\begin{xy}
\xymatrix{
\mathcal X'\ar[r]\ar[d]&\mathcal X\ar[d]^f\\
\mathcal Y'\ar[r]&\mathcal Y
}
\end{xy}
$$

\end{lem}

\begin{lem}\label{blowup of fib}
Let 
$$(\mathcal X_d\to\mathcal X_{d-1}\to\cdots\to\mathcal X_1\to\mathcal X_0,Z_0,\{\sigma_{ij}\})$$
be a $G$-split nodal fibration datum such that $\mx_0$ is regular and $Z_0$ is a divisor with $G$-strict normal crossings. 
Then, there exist a sequence of $G$-equivariant blowups 
\begin{eqnarray*}
\mx'_1\to&\mx_1,\\
\mx'_2\to&\mx_2\times_{\mx_1}\mx'_1,\\
\cdots&\\
\mx'_d\to&\mx_d\times_{\mx_{d-1}}\mx'_{d-1}
\end{eqnarray*}
such that 
the center of $\mx'_i\to\mx_i\times_{\mx_{i-1}}\mx'_{i-1}$ is outside the smooth locus of $\mx_i\times_{\mx_{i-1}}\mx'_{i-1}\to\mx'_{i-1}$ 
and that 
$$(\mx'_d\to\mx'_{d-1}\to\cdots\to\mx'_1\to\mx_0,Z_0,\{\sigma_{ij}\})$$ 
is a strictly $G$-split nodal fibration datum. 
\end{lem}

\begin{proof}
The $d=1$ case is proved in \cite[Lemme 4.4.4]{V1}. 
We argue by induction on the relative dimension $d$. 
If $d=0$, we have nothing to do. 
We assume that the assertion holds for $G$-split nodal fibration data of relative dimension $d-1$. 
Then, we may assume that 
$$(\mathcal X_{d-1}\to\cdots\to\mathcal X_1\to\mathcal X_0,Z_0,\{\sigma_{ij}\})$$
is a strictly $G$-split nodal fibration datum. 
Then, since $\mx_{d-1}$ is regular and $Z_{d-1}$ is a divisor with $G$-strict normal crossings, we can apply the $d=1$ case to the $G$-split nodal curve $\mx_d\to\mx_{d-1}$. 
\end{proof}

	\begin{cor}\label{Lipman}
	Let $S$ be an excellent noetherian scheme which is irreducible and of dimension $\leq2$. 
Let $X$ be an $S$-scheme of finite type with an action of a finite group $G$ 
and $Z$ a proper closed subset of $X$ which is $G$-stable. 
We assume that the generic geometric fiber of $X\to S$ is irreducible. 
Then, there exists a Galois alteration $(\phi,\alpha):(X',G')\to(X,G)$ with
$X'$ regular and a divisor with $G'$-strict normal crossings containing the pullback $\phi^{-1}(Z)$. 
	\end{cor}
	\begin{proof}
We consider the trivial action of $G$ on $S$. 
By Lemma \ref{alt}, we may assume that $X\to S$ is given by a $G$-split nodal fibration datum 
$(X=X_d\to\cdots\to X_0=S,Z_0,\{\sigma_{ij}\})$ 
with $Z=Z_d$ in the notation in Definition \ref{datum}. 
Here, by \cite{L}, we may assume that $S$ is regular. 
Further, by \cite[Corollary 0.4]{cjs}, we may assume that $Z_0$ is a divisor of $S_0$ with simple normal crossings. 
Then, the assertion follows from Lemma \ref{blowup of fib}.
	\end{proof}

\section{Log structures and nodal curves}\label{log}
We refer to \cite{KK} for log structures. 
In particular, log structures are always considered on the \'etale sites of schemes. 
For a log structure $\mm_X\to\mo_X$ on a scheme $X$, 
we denote the quotient $\mm_X/\mo_X^*$ by $\mbar_X$. 

Let $X$ be a regular noetherian scheme 
and $D$ a divisor with simple normal crossings. 
We consider the log structure $\mm_X\to\mo_X$ defined by $D$, that is, 
$\mm_X=j_*\mo^*_U\cap\mo_X$ 
where $U=X\setminus D$ and $j$ is the open immersion $U\to X$. 
We write $D$ as the sum of the irreducible components: 
$D=\bigcup_{i=1}^nD_i$. 
For a subset $I\subset \{1,\ldots,n\}$, we put 
$D_I=\bigcap_{i\in I}D_i$ 
and 
$\mathring D_I=D_I\setminus\bigcup_{i\notin I}D_{I\cup\{i\}}$. 

\begin{lem}\label{log of snc}
In the above settings, we have a canonical isomorphism 
$\mbar_X|_{\mathring D_I}\cong\nn^I$. 
\end{lem}
\begin{proof}
In the case where $X=\spec A$ is affine and $D$ is defined by $h_1\cdots h_n$ for some prime elements $h_i\in A$ ($i=1,\ldots,n$), 
the composite of the map 
$\nn^I\to\mm_X|_{\mathring D_{I}}:e_i\mapsto h_i$ 
and the natural map $\mm_X|_{\mathring D_{I}}\to\mbar_X|_{\mathring D_{I}}$ 
is an isomorphism. 
Note that it does not depend on the choice of the prime elements $h_i$. 
Thus, in general case, by gluing we get the desired isomorphism. 
\end{proof}

Further, the isomorphism in Lemma \ref{log of snc} is functorial in the following sense (Lemma \ref{functoriality log}): 
We consider a commutative diagram 
$$
\begin{xy}
\xymatrix{
X\ar[r]^f&Y\\
U\ar[r]\ar[u]^j&V\ar[u]_{j'}
}
\end{xy}
$$
such that 
$X$ and $Y$ are regular, 
$f$ is flat, 
$j$ and $j'$ are open immersions, 
and $D=X\setminus U$ and $E=Y\setminus V$ are divisors with simple normal crossings. 
We consider the log structure $\mm_X$ (resp. $\mm_Y$) 
defined by $D$ (resp. $E$). 
Then, we have a natural morphism $(X,\mm_X)\to(Y,\mm_Y)$ of log schemes. 
We write $D$ and $E$ as the sums of their irreducible components: 
$D=\sum_{i\in I}D_i$ and $E=\sum_{j\in J}E_j$. 
Let $I_f$ be the subset of $I$ consisting of elements $i\in I$ such that $f(D_i)\subset E$. 
Here, since $f$ is flat, 
we can consider 
the pullback $f^*E$ of the divisor $E$
and 
a map $\varphi:I_f\to J$ sending $i$ to $j$ such that $f(D_i)\subset E_j$. 
We write the divisor $f^*E$ 
as $\sum_{i\in I}m_iD_i$ for some $m_i\in\zz_{\geq0}$. 
\begin{lem}\label{functoriality log}
In the above notations, 
for every subset $I'\subset I$, we have a commutative diagram 
$$
\begin{xy}
\xymatrix{
(f^{-1}\mbar_Y)|_{\mathring D_{I'}}
\ar[r]&
\mbar_X|_{\mathring D_{I'}}
\\
\nn^{\varphi(I')}\ar[r]\ar[u]^{\cong}&\nn^{I'}\ar[u]_{\cong},
}
\end{xy}
$$
where 
the lower horizontal arrow is the morphism sending 
$e_j$ to $\sum_{i\in I'\cap \varphi^{-1}(j)}m_ie_i$.
\end{lem}
\begin{proof}
This follows from the fact that $f^*E_j$ is locally defined by the ideal generated by $\prod_{i\in \varphi^{-1}(j)}h_i^{m_i}$, where $h_i$ is a defining equation of $D_i$. 
\end{proof}
	
\begin{lem}[{\rm\cite[2.23]{dJ}}]\label{local description}
Let $A$ be a strictly henselian noetherian local ring. 
Let $f:\mx\to\spec A$ be a split nodal curve. 
Then, \'etale locally on $\mx$, there exists an \'etale $A$-morphism 
$\mx\to\spec A[x,y]/(xy-t)$ 
for some $t\in A$. 
\end{lem}

\begin{lem}\label{nodal log}
Let $\mx$ and $\my$ be regular noetherian schemes. 
Let $f:\mx\to\my$ be a split nodal curve, 
$Z$ a divisor of $\my$ with simple normal crossings 
such that $f$ is smooth over $\my\setminus Z$, 
and $\sigma:\my\to\mx$ a section of $f$ into the smooth locus. 
We consider the divisor 
$D=f^{-1}(Z)\cup\sigma(\my)$. 
We consider the log structure $\mm_\mx$ $ ($resp. $\mm_\my$) on $\mx$ $($resp. $\my)$ defined by $D$ $($resp. $Z)$. 
Then, the morphism $(\mx,\mm_\mx)\to(\my,\mm_\my)$ of log schemes is 
log smooth and saturated.
\end{lem}
For the definition of saturated morphisms of log schemes, see \cite[I.3.5, I.3.7, I.3.12, II.2.10]{T}. 
\begin{proof}
Since the question is \'etale local, 
we may assume, by Lemma \ref{local description}, that 
$\my$ is of the form $\spec A$ for some ring $A$ 
and there exists an $A$-isomorphism $\mx\to\spec A[x,y]/(xy-t)$ for some nonzero $t\in A$. 
Then, the diagram 
$$
\begin{xy}
\xymatrix{
\mx\ar[rr]^{\phi}\ar[d]_f&&\spec\zz[\nn^2]\ar[d]^h\\\my\ar[rr]^{\psi}&&\spec\zz[\nn],}
\end{xy}
$$
is a Cartesian diagram of log schemes, 
where the morphism $h$ is defined by $\nn\to\nn^2:1\mapsto(1,1)$ and 
$\psi$ (resp. $\phi$) is defined by $\nn\to A:1\mapsto t$ (resp. $\nn^2\to A[x,y]/(xy-t):(1,0)\mapsto x,\ (0,1)\mapsto y$). 
Then $f$ is log smooth, since $\zz\to\zz^2:1\to(1,1)$ is injective and its cokernel is torsion free. 
Further $f$ is integral (see \cite[Proposition 4.1]{KK} for the definition), since $h$ is flat. 
Since every fiber of $f$ is reduced, we can use \cite[Theorem 4.2]{T} to see that $f$ is saturated. 
\end{proof}


\section{Monodromy action composed with a finite group action}\label{G-sheaves}

Let $Y$ be a noetherian scheme, 
$\ell$ be a prime number invertible on $Y$, 
and $\mf$ be a locally constant constructible sheaf of $\zz/\ell^n$-modules on $Y$. 
We consider an admissible action of a finite group $G$ on the scheme $Y$ and 
a $G$-sheaf structure on $\mf$, 
that is a family of morphisms $\{\varphi_g:g^*\mf\to\mf\}_{g\in G}$ satisfying the cocycle condition; 
for $g,h\in G$, the composite $(gh)^*\mf\cong h^*g^*\mf\overset{h^*\varphi_g}{\to}h^*\mf\overset{\varphi_h}{\to}\mf$ coincides with $\varphi_{gh}$. 
We assume that $Y$ is connected and that $Y\to Y_0=Y/G$ is \'etale, and hence a Galois finite \'etale cover.  
We denote the Galois group $\Aut(Y/Y_0)$ by $H$ 
and take a geometric point $y$ of $Y$. 
Then, we have a natural action of 
$G\times_H\pi_1(Y_0,y)$ 
on the stalk $\mf_y$ defined as follows: 
For a pointed connected finite \'etale cover $(Y',y')$ of $(Y,y)$ which is Galois over $(Y_0,y)$, 
an element $(g,\sigma)\in G\times_H\pi_1(Y_0,y)$ induces 
a commutative diagram 
$$
\begin{xy}
\xymatrix{
Y'\ar[r]^\sigma\ar[d]_p&Y'\ar[d]^p\\ Y\ar[r]^g&Y.
}
\end{xy}
$$
We define an action of $(g,\sigma)$ on $\Gamma(Y',\mf)$ by the composite 
$
\Gamma(Y',\mf)\cong\Gamma(Y',p^*\mf)\overset{\sigma^*}{\to}\Gamma(Y',\sigma^*p^*\mf)\cong\Gamma(Y',g^*\mf)\overset{\varphi_g}{\to}\Gamma(Y',\mf).
$
It defines an action of 
$G\times_H\pi_1(Y_0,y)$ 
on $\Gamma(Y',\mf)$ which is functorial with respect to $(Y',y')$, 
and hence we obtain an action 
$G\times_H\pi_1(Y_0,y)$ 
on $\mf_y\cong\varinjlim_{(Y',y')}\Gamma(Y',\mf)$, 
where $(Y',y')$ runs over all pointed connected finite \'etale covers which are Galois over $(Y_0,y)$. 

This action is defined also in $\ell$-adic settings: 
Let $\mf$ be a lisse $\zl$-sheaf on $Y$ with a $G$-sheaf structure. 
We write $\mf$ as an inverse system $(\mf_n)_{n\geq1}$, 
where $\mf_n$ is a locally constant constructible sheaf of $\zz/\ell^n$-modules 
such that $\mf_{n+1}\otimes_{\zz/\ell^{n+1}}\zz/\ell^n\cong\mf_n$. 
Then, we have a natural action of 
$G\times_H\pi_1(Y_0,y)$ 
on the stalk $\mf_y\cong\varprojlim_n(\mf_{n,y})$. 
We also have, for a lisse $\ql$-sheaf $\mf$ on $Y$ with a $G$-sheaf structure, a natural action of 
$G\times_H\pi_1(Y_0,y)$ on the stalk $\mf_y$. 

Let $\my$ be the spectrum of a strictly henselian regular local ring 
with an action of a finite group $G$. 
Let $Z$ be a divisor of $\my$ with simple normal crossings which is $G$-stable. 
We denote by $\mm_\my$ the log structure on $\my$ defined by $Z$. 
Let $\ell$ be a prime number invertible on $\my$. 
Let $\mf$ be a constructible sheaf of $\zz/\ell^n$-modules on the Kummer \'etale topos $\my_{\rm\text{k\'et}}$ 
with a $G$-sheaf structure. 
See \cite[Section 2]{I} for the details of Kummer \'etale toposes. 
We denote the fraction field of $\my$ by $K$. 
We write $K_0$ for the fixed subfield $K^G$ by $G$ 
and $H$ for the Galois group $\gal(K/K_0)$. 
We take 
a log-geometric point $\widetilde y$ over the log point $(y,\mm_y)$, where $y$ is the closed point of $\my$ and the log structure $\mm_y$ is the pullback of the log structure $\mm_\my$ by the closed immersion $y\to\my$. 
Let $(\widetilde\my,\mm_{\widetilde\my})$ be the log strict localization \cite[4.5]{I} of $(\my,\mm)$ at $\widetilde y$. 
We write $\widetilde K$ for the fraction field of $\widetilde \my$. 
\begin{lem}\label{gal}
$\widetilde K$ 
is Galois over 
$K_0$. 
\end{lem} 
\begin{proof}
Let $\Kbar$ be a separable closure of $K$. 
We fix an embedding $\widetilde K\to\Kbar$. 
Then $\widetilde K$ is 
the union of all finite sub-extensions $K'$ of $\Kbar/K$ 
such that $K'$ is isomorphic over $\my$ to the fraction field of some connected scheme with a log structure which is finite and Kummer \'etale over $\my$. 
Thus, for every $\sigma\in\gal(\Kbar/K_0)$, 
we have $\sigma(\widetilde K)\subset\widetilde K$. 
\end{proof}
We define a natural action 
$G\times_H\gal(\widetilde K/K_0)$ 
on the stalk $\mf_{\tilde y}$. 
Let $(\my',\widetilde y')$ be a pointed connected finite Kummer \'etale cover of $(\my,\widetilde y)$ 
such that the fraction field $K'$ of $\my'$ is Galois over $K_0$. 
Note that $K'$ can be viewed as a subfield of $\widetilde K$ by the natural morphism $\my'\to\widetilde\my$. 
Then, an element $(g,\sigma)\in G\times_H\gal(\widetilde K/K_0)$ 
induces a commutative diagram 
$$
\begin{xy}
\xymatrix{
\my'\ar[r]^\sigma\ar[d]_p&\my'\ar[d]^p\\ \my\ar[r]^g&\my.
}
\end{xy}
$$
We define an action of $(g,\sigma)$ on $\Gamma(\my',\mf)$ by the composite 
$
\Gamma(\my',\mf)\cong\Gamma(\my',p^*\mf)\overset{\sigma^*}{\to}\Gamma(\my',\sigma^*p^*\mf)\cong\Gamma(\my',g^*\mf)\overset{\varphi_g}{\to}\Gamma(\my',\mf).
$
It defines an action of 
$G\times_H\gal(\widetilde K/K_0)$ 
on $\Gamma(\my',\mf)$ which is functorial with respect to $(\my',\widetilde y')$. 
Now we have 
$\mf_{\tilde y}\cong\varinjlim_{(\my',\widetilde y')}\Gamma(\my',\mf)$, 
where $(\my',\widetilde y')$ runs over all pointed finite Kummer \'etale covers of $(\my,\widetilde y)$. 
Further, by Lemma \ref{gal}, in the filtered inverse system of pointed finite Kummer \'etale covers of $(\my,\widetilde y)$, 
pointed finite Kummer \'etale covers $(\my',\widetilde y')$ such that $\spec K\times_\my\my'$ is Galois over $\spec K_0$ form 
a cofinal system. 
Thus, we obtain an action 
$G\times_H\gal(\widetilde K/K_0)$ 
on $\mf_{\tilde y}$. 
As before, this action is defined also in $\ell$-adic settings. 

\begin{lem}\label{equiv}
Let $\Kbar$ be a separable closure of $K$. We fix an embedding $\iota:\widetilde K\to\Kbar$. 
Then the cospecialization map 
$\mf_{\tilde y}\to\mf_{\Kbar}$ 
induced by $\iota$ 
is $G\times_H\gal(\Kbar/K_0)$-equivariant, where on $\mf_{\tilde y}$ we consider the action defined through the surjection $G\times_H\gal(\Kbar/K_0)\to G\times_H\gal(\widetilde K/K_0)$ induced by $\iota$. 
\end{lem}
\begin{proof}
It suffices to show that, in the case of torsion coefficients, 
for every finite Kummer \'etale cover $\my'$ of $\my$ such that 
the fraction field $K'$ is Galois over $K_0$, 
the morphism 
$\Gamma(\my',\mf)
\to
\Gamma(\spec K',\mf)$ 
is $G\times_H\gal(K'/K_0)$-equivariant. 
This follows from the following commutative diagram (we write $Y'$ for $\spec K'$ to simplify the notations):
$$
\begin{xy}
\xymatrix{
\Gamma(\my',\mf)\ar@{=}[r]\ar[d]
&
\Gamma(\my',p^*\mf)\ar[r]^{\sigma^*}\ar[d]
&
\Gamma(\my',\sigma^*p^*\mf)\ar@{=}[r]\ar[d]
&
\Gamma(\my',g^*\mf)\ar[r]^{\varphi_g}\ar[d]
&
\Gamma(\my',\mf)\ar[d]
\\
\Gamma(Y',\mf)\ar@{=}[r]
&
\Gamma(Y',p^*\mf)\ar[r]^{\sigma^*}
&
\Gamma(Y',\sigma^*p^*\mf)\ar@{=}[r]
&
\Gamma(Y',g^*\mf)\ar[r]^{\varphi_g}
&
\Gamma(Y',\mf).
}
\end{xy}
$$

\end{proof}

\section{$\ell$-independence}\label{indep}

We consider a monodromy action composed with a finite group action defined as follows. 
Let $K$ be a field with an action of a finite group $G$. 
Let $X$ be a scheme separated of finite type over $K$ with an action of $G$ which makes the structural morphism $X\to\spec K$ equivariant. 
Let $\ell$ be a prime number different from the characteristic of  $K$. 
We write $H$ for the Galois group $\gal(K/K_0)$, where $K_0$ is the fixed subfield $K^G$ of $K$. 
We take a separable closure $\Kbar$ of $K$. 
Then, as in section \ref{G-sheaves}, we have a natural continuous action of the fiber product $G\times_H\gal(\Kbar/K_0)$ on $H^q_c(X\times_K\Kbar,\mathbb Q_\ell)$ for every $q$. 

\begin{thm}\label{ellindepmonod}
Let $S$ be an excellent noetherian scheme of dimension $\leq2$. 
Let $K$ be a field with an action of a finite group $G$ 
and $\spec K\to S$ a morphism which is essentially of finite type (in the sense of Definition \ref{cptfn}.\ref{fg}). 
Let $X$ be a scheme separated of finite type over $K$ with an admissible action of $G$ which makes the structural morphism $X\to\spec K$ equivariant. 
We take a separable closure $\overline K$ of $K$ and a prime number $\ell$ invertible on $S$. 
We write $H$ for the Galois group $\gal(K/K_0)$, where $K_0$ is the fixed subfield $K^G$ of $K$. 
Then, for every element $(g,\sigma)\in G\times_HE_{K_0/S}$, 
\begin{enumerate}
\item
the eigenvalues of the action of $(g,\sigma)$ on $H^q_c(X\times_K\Kbar,\mathbb{Q}_\ell)$, for each $q$, are roots of unity, 
\item
the alternating sum of the traces 
$$\sum_q(-1)^q\tr((g,\sigma),H^q_c(X\times_K\Kbar,\mathbb{Q}_\ell))$$
is an integer independent of a prime number $\ell$ invertible on $S$. 
\end{enumerate}
\end{thm}
\begin{rem}
Note that by \cite[Theorem 2.2]{IZ}, the alternating sum of the traces of $(g,\sigma)$ acting on $H_c^i(X_{\Kbar},\ql)$ does not change even if we replace $H_c^i$ by $H^i$. 
\end{rem}

In the proof of the theorem, we use the following terminology: 
\begin{defn}
Let $S$ be a noetherian scheme. 
Let $K$ be a field with an action of a finite group $G$ 
and $\spec K\to S$ a morphism which is essentially of finite type. 
Let $X$ be a scheme separated of finite type over $K$ with an admissible action of $G$ which makes the structural morphism $X\to\spec K$ equivariant. 
A $G$-{\it integral model} over $S$ of $X\to\spec K$ is 
a $G$-equivariant morphism $\mx\to\my$ of normal connected proper $S$-schemes together with a $G$-equivariant commutative diagram 
$$
\begin{xy}
\xymatrix{X\ar[r]\ar[d]&\mx\ar[d]\\\spec K\ar[r]&\my}
\end{xy}
$$
such that 
the lower horizontal arrow is a compactification of $K$ over $S$ in the sense of Definition \ref{cptfn} 
and that 
the natural morphism $X\to\mx\times_\my\spec K$ is a dense open immersion. 
\end{defn}
\begin{lem}\label{nagata}
Let $S$ and $X\to\spec K$ be as above. 
Then a $G$-integral model of $X\to\spec K$ over $S$ exists. 
\end{lem}
\begin{proof}
By Nagata's compactification theorem, 
we can find an $e$-integral model $\mx_1\to\my_1$ of $X\to\spec K$, where $e$ is the trivial group. 
We write $\iota_X$ (resp. $\iota_Y$) for the morphism $X\to\mx_1$ (resp. $\spec K\to\my_1$). 
Take the closure of the image of the $G$-equivariant morphism $X\to\prod_{g\in G}\mx_1:x\mapsto(\iota_X(gx))_{g\in G}$ 
(resp. $\spec K\to\prod_{g\in G}\my_1:y\mapsto(\iota_Y(gy))_{g\in G}$), where the products are taken over $S$.  
\end{proof}

\begin{proof}[Proof of Theorem \ref{ellindepmonod}]
We prove the assertions by induction on $d=\dim X$. 
By the arguments as in \cite[2.2.3]{V}, the proof is reduced to the case where 
$X$ is normal connected and geometrically irreducible over $K$. 
For the completeness, we include the reduction arguments. 
In the reduction arguments, we use the following two claims frequently. 
\begin{claim}\label{qgal}
Let $K'$ be a finite quasi-Galois extension over $K$ which is also quasi-Galois over $K_0$. 
We write $H'$ for $\Aut(K'/K_0)$. 
Then, 
the assertions for $X'=(X\times_KK')_{\mathrm{red}}\to\spec K'$ with the $G'=G\times_HH'$-action 
are equivalent to those for $X\to\spec K$ with the $G$-action. 
\end{claim}
\begin{proof}
Let $\overline{K'}$ be a separable closure of $K'$ containing $\Kbar$ 
and let $K'_0$ be the fixed subfield $(K')^{H'}$. 
Then, we have a natural isomorphism 
$H^q_c(X'\times_{K'}\overline{K'},\mathbb{Q}_\ell)
\cong
H^q_c(X\times_K\Kbar,\mathbb{Q}_\ell)$, 
which is equivariant with respect to the map 
$\alpha:G'\times_{H'}E_{K'_0/S}\to G\times_HE_{K_0/S}$. 
Since $K'_0$ is purely inseparable over $K_0$, the map $\alpha$ is bijective by Lemma \ref{on E}.\ref{ram part pinsep}. 
Thus, the assertion follows. 
\end{proof}
\begin{claim}\label{shrink}
Let $U$ be a $G$-stable dense open subscheme of $X$. 
Then, under the induction hypothesis, the assertions for $U$ and those for $X$ are equivalent. 
\end{claim}

We can take a finite Galois extension $K'$ over $K$ which is also Galois over $K_0$ 
such that every irreducible component of $X\times_KK'$ 
is geometrically irreducible over $K'$. 
We write $H'$ for the Galois group of $K'/K$. 
Then, by Claim \ref{qgal}, 
the assertions for $X'=(X\times_KK')_{\mathrm{red}}$ with the $G'=G\times_HH'$-action 
are equivalent to those for $X$ with the $G$-action. 
Thus, we may assume that $X$ is reduced and every irreducible component of $X$ is geometrically irreducible over $K$. 
Further by Claim \ref{shrink}, we may assume that $X$ is normal. 

Let $\{X_i\}_{i=1}^r$ be the set of connected components of $X$. 
We may assume that $g$ transitively permutes the irreducible components. 
To show the assertion 1, it suffices to show that the eigenvalues of the $r$-th iteration of $(g,\sigma)$ acting on $H^q_c(X_{i\Kbar},\ql)$ are roots of unity. 
For the assertion 2, we have 
$$\tr((g,\sigma),H^q_c(X_{\Kbar},\ql))=\sum_{\substack{i\in I\\ gX_i=X_i}}\tr((g,\sigma),H^q_c(X_{i\Kbar},\ql)).$$
Thus, we may further assume that $X$ is connected.

As in \cite[2.2.3]{V}, 
we will reduce the proof to the case where $X\to\spec K$ admits a $G$-integral model $\mx\to\my$ over $S$ which is a nodal fibration. 

We take any $G$-integral model $\mx\to\my$ of $X\to\spec K$ over $S$, which exists by Lemma \ref{nagata}. 
By Lemma \ref{alt}, we can find 
projective Galois alterations 
$(\mx',G')\to(\mx,G)$ 
and
$(\my',G')\to(\my,G)$ 
with a commutative diagram 
$$
\begin{xy}
\xymatrix{
\mathcal X'\ar[r]\ar[d]_{f'}&\mathcal X\ar[d]^{f}\\
\mathcal Y'\ar[r]&\mathcal Y
}
\end{xy}
$$
such that $f'$ is a $G'$-split nodal fibration. 
Let $K'$ be the function field of $\my'$. 
Applying Claim \ref{replace by alt} below to the diagram 
$$
\begin{xy}
\xymatrix{
X\times_\mx\mx'\ar[r]\ar[d]&X\ar[d]\\
\spec K'\ar[r]&\spec K,
}
\end{xy}
$$
we may assume that 
there exists a $G$-integral model of $X\to\spec K$ over $S$ which is a $G$-split nodal fibration. 


\begin{claim}\label{replace by alt}
Let $(X',G')\to(X,G)$ be 
a Galois alteration with a $G'$-equivariant commutative diagram 
$$
\begin{xy}
\xymatrix{
X'\ar[r]\ar[d]&X\ar[d]\\
\spec K'\ar[r]&\spec K
}
\end{xy}
$$
such that $K'$ is a quasi-Galois extension of $K$ which is also quasi-Galois over $K_0$. 
Then, under the induction hypothesis, the assertions for $X'\to\spec K'$ with the $G'$-action imply those for $X\to\spec K$ with the $G$-action. 
\end{claim}
\begin{proof}
By Claim \ref{qgal}, we may assume that $K'=K$. 
By Claim \ref{shrink}, we may assume that $X'\to X$ is finite. We write $\Gamma$ for the kernel of the map $G'\to \Aut(X)$. 
Then, by the definition of Galois alterations, 
$X'/\Gamma\to X$ is a finite radicial surjection. 
Then, by Lemma \ref{coh of quot} below, we have 
$
H^q_c(X_{\Kbar},\ql)
\cong
H^q_c(X'_{\Kbar},\ql)^\Gamma
$. 
Thus, the assertion 1 of Theorem \ref{ellindepmonod} for $X'\to\spec K'$ with the $G'$-action implies the assertion 1 for $X\to\spec K$ with the $G$-action. 
Further, we have 
$$
\tr((g,\sigma),H^q_c(X\times_{K}\Kbar,\mathbb Q_\ell))
=
\frac{1}{|\Gamma|}\sum_{g'}\tr((g',\sigma),H^q_c(X'\times_{K}\Kbar,\mathbb Q_\ell)).
$$
Here, $g'$ runs elements of $G'$ which define the same element in $\Aut(X)$ as $g$. 
Thus, the assertion 2 of Theorem \ref{ellindepmonod} for $X'\to\spec K'$ with the $G'$-action also implies the assertion 2 for $X\to\spec K$ with the $G$-action. 
\end{proof}
\begin{lem}\label{coh of quot}
Let $K$ be a field and $X$ a scheme separated of finite type over $K$. 
Let $G$ be a finite group acting admissibly on $X$ over $K$. 
Then, the natural morphism 
$$
H^q_c((X/G)_{\bar K},\ql)
\to
H^q_c(X_{\bar K},\ql)^G
$$
is an isomorphism, where $\bar K$ is a separable closure of $K$ and $\ell$ is a prime number distinct from the characteristic of $K$. 
\end{lem}
\begin{proof}
See the proof of \cite[Lemma 2.3]{O}. 
\end{proof}


Further, we reduce the proof to the case where there exists a $G$-integral model over $S$ which is a strictly $G$-split nodal fibration (Definition \ref{fib}). 

We take a $G$-integral model $\mx\to\my$ of $X\to\spec K$ over $S$ which is a $G$-split nodal fibration 
and  a $G$-split nodal fibration datum 
$(\mx=\mx_d\to\cdots\to\mx_0,Z_0,\{\sigma_{ij}\})$ 
realizing $\mx\to\my$. 
By Corollary \ref{Lipman}, we can find a Galois alteration $(\psi,\alpha):(\my',G')\to(\my,G)$ with $\my'$ regular and a divisor $Z_0'$ with $G'$-strict normal crossings containing $\psi^{-1}(Z_0)$. 
Then, by Lemma \ref{blowup of fib}, we can find a $G'$-equivariant commutative diagram 
$$
\begin{xy}
\xymatrix{
\mathcal X'\ar[r]\ar[d]&\mathcal X\ar[d]\\
\mathcal Y'\ar[r]&\mathcal Y
}
\end{xy}
$$
such that $\mx'\to\my'$ is a strictly $G$-split nodal fibration. 
Thus, by Claim \ref{replace by alt}, we may assume that $X\to\spec K$ admits a $G$-integral model $\mx\to\my$ over $S$ which is a strictly $G$-split nodal fibration. 

Let $(g,\sigma)\in G\times_HE(K_0/S)$. 
We may assume that $G$ is a cyclic group and is generated by $g$. 
Then, by Lemma \ref{char of E}, 
there exists a geometric point $y$ of $\my$ fixed by $g$, 
and hence, $G$ naturally acts on the strict localization $\my_{(y)}$. 
Since, by Claim \ref{shrink}, we may assume that 
$X=(\mx\setminus Z_d)\times_\my\spec K$, 
the assertion follows from Proposition \ref{localmonod} below. 
\end{proof}

\begin{prop}\label{localmonod}
Let $\my$ be the spectrum of an excellent regular strictly local ring with an action of a finite group $G$ which is trivial on the residue field of $\my$. 
We consider a strictly $G$-split proper nodal fibration datum 
$(\mx=\mx_d\to\cdots\to\mx_0,Z=Z_0,\{\sigma_{ij}\})$ 
with $\mx_0=\my$. 
Let $K$ be the function field of $\my$ with a separable closure $\Kbar$. 
Let $\ell$ be a prime number invertible on $\my$. 
We write $H$ for the Galois group of $K$ over the fixed subfield $K_0=K^G$. 
Let $Z_d$ be the closed subset of $\mx=\mx_d$ as in Definition \ref{datum}. 
Then, for every $(g,\sigma)\in G\times_H\gal(\Kbar/K_0)$, 
\begin{enumerate}
\item
the eigenvalues of the action of $(g,\sigma)$ on $H^q((\mx\setminus Z_d)\times_\my\spec\Kbar,\mathbb Q_\ell)$, for each $q$, are roots of unity, 
\item
the alternating sum  
\begin{eqnarray}\label{alt sum}
\sum_q(-1)^q\tr((g,\sigma),H^q((\mx\setminus Z_d)\times_\my\spec\Kbar,\mathbb Q_\ell))
\end{eqnarray}
is an integer independent of $\ell$. 
\end{enumerate}
\end{prop}



\begin{proof}
Let $\mm_{\mx_i}$ be  the log structure on $\mx_i$ defined by $Z_i$, 
where $Z_i$ is the closed subscheme defined as in Definition \ref{datum}. 
By purity for log regular log schemes \cite[Theorem 7.4]{I}, 
we have a canonical isomorphism 
$$H^q((\mx\setminus Z_d)\times_\my\spec\Kbar,\mathbb Q_\ell))\cong
H^q(\mx_{\Kbar,\text{k\'et}},\ql)
$$ 
for each $q$. 
Here, $\mx_{\Kbar,\text{k\'et}}$ denotes the Kummer \'etale topos of the log scheme 
$(\mx,\mm_\mx)\times_{(\my,\mm_\my)}\spec\Kbar$, 
where we regard $\spec\Kbar$ as a log scheme with the trivial log structure. 

By Lemma \ref{nodal log}, 
$f^{\log}:(\mx,\mm_\mx)\to(\my,\mm_\my)$ induced by the composite $\mx\to\my$ 
is a log smooth and saturated morphism of fs log schemes. 
By a proper log smooth base change theorem \cite[Proposition 4.3]{N}, $Rf^{\text{log}}_*\mathbb Q_\ell$ is lisse with respect to the Kummer \'etale topology. 
Here, for each $q$ the lisse sheaf $R^qf^{\log}_*\ql$ has a canonical $G$-sheaf structure. 
Thus, using the proper exact base change theorem \cite[Theorem 5.1]{N} and the construction in Section \ref{G-sheaves}, 
we obtain a $G\times_H\gal(\Kbar/K_0)$-equivariant isomorphism 
$$H^q(\mx_{\Kbar,\text{k\'et}},\ql)\cong H^q(\mx_{\tilde y,\text{k\'et}},\ql).$$
Here, $\mx_{\widetilde y,\text{k\'et}}$ denotes 
the log scheme $(\mx,\mm_\mx)\times_{(\my,\mm_\my)}\widetilde y$, 
where $\widetilde y$ is a log geometric point over $(y,\mm_\my|_y)$ and 
where the fiber product is taken in the category of saturated log schemes. 

We denote the underlying scheme of the log scheme $\mx_{\tilde y}$ 
by $(\mx_{\tilde y})^o$. 
Since $f^{\log}$ is saturated, the natural morphism 
$(\mathcal X_{\tilde y})^o\to\mathcal X_y$ 
of usual schemes is an isomorphism by \cite[Proposition II.2.13]{T}. 
We consider the natural morphism 
$\varepsilon:\mathcal X_{\tilde y}\to\mathcal X_{\tilde y}^o$ 
of log schemes. 
We identify it with the morphism 
$\mathcal X_{\tilde y}\to\mathcal X_y$, where we regard $\mx_y$ as a log scheme with the trivial log structure. 
Then, the sheaf $R^q\varepsilon_*\ql$ for each $q$ has 
a canonical $G\times_H\gal(\Kbar/K_0)$-sheaf structure and 
we have a $G\times_H\gal(\Kbar/K_0)$-equivariant spectral sequence 
\begin{eqnarray}\label{spec seq}
E_2^{p,q}=H^p(\mx_y,R^q\varepsilon_*\ql)\Rightarrow H^{p+q}(\mx_{\tilde y,\text{k\'et}},\ql).
\end{eqnarray}
Thus, we have an equality 
$$
\sum_q(-1)^q\tr((g,\sigma),H^q(\mx_{\tilde y,\text{k\'et}},\ql))
=
\sum_q(-1)^q\sum_p(-1)^p\tr((g,\sigma),H^p(\mx_y,R^q\varepsilon_*\ql)).
$$

We shall describe $R^q\varepsilon_*\ql$ in terms of log structures. 
Let $\mgp_{\mx_y/y}$ be the cokernel of the morphism 
$f^{-1}\mgp_y\to\mgp_{\mx_y}$. 
It has a natural $G$-sheaf structure 
coming from the action of $G$ on the log schemes $\mx_y$ and $y$. 
By \cite[Corollaire 5.4]{V}, we have
a canonical $G\times_H\gal(\Kbar/K_0)$-equivariant isomorphism 
$$
R^q\varepsilon_*\ql
\cong
\bigwedge^q(\mgp_{\mx_y/y}\otimes\ql(-1)),
$$
Here, we consider the $G\times_H\gal(\Kbar/K_0)$-sheaf structure on 
$\mgp_{\mx_y/y}$ defined by the surjection 
$G\times_H\gal(\Kbar/K_0)\to G$ and the $G$-sheaf structure on it. 
Thus, the eigenvalues of $(g,\sigma)$ acting on $E_2^{p,q}$ in the spectral sequence (\ref{spec seq}) are roots of unity, and hence, the assertion 1 follows. 
Further, we have an equality 
$$
\sum_p(-1)^p\tr((g,\sigma),H^p(\mx_y,R^q\varepsilon_*\ql))
=
\sum_p(-1)^p\tr(g,H^p(\mx_y,\bigwedge^q(\mgp_{\mx_y/y}\otimes\ql)))
$$
for each $q$. Then, the assertion follows from Proposition \ref{ellindep geom} below. 
\end{proof}

\begin{prop}\label{ellindep geom}
Let $\my$ be the spectrum of an excellent regular strictly local ring with an action of a finite group $G$ which is trivial on the residue field of $\my$. 
We consider a $G$-equivariant commutative diagram 
$$
\begin{xy}
\xymatrix{
\mx\ar[r]^f&\my\\
U\ar[r]\ar[u]^j&V\ar[u]_{j'}
}
\end{xy}
$$
such that 
$\mx$ is regular, 
$f$ is flat and proper, 
$j$ and $j'$ are open immersions, 
and $D=\mx\setminus U$ and $Z=\my\setminus V$ are divisors with simple normal crossings. 
We endow $\mx$ (resp. $\my$) with the log structure $\mm_\mx$ (resp. $\mm_\my$) defined by $D$ (resp. $Z$). 
Let $y$ be the closed point of $\my$ 
and let $\mgp_{\mx_y/y}$ be the cokernel of the morphism 
$f^{-1}\mgp_y\to\mgp_{\mx_y}$ with a natural action of $G$, 
where $\mm_{\mx_y}$ (resp. $\mm_y$) is the log structure on $\mx_y$ (resp. $y$) obtained from $\mm_\mx$ (resp. $\mm_\my$) via pullback. 
Then, for every $g\in G$ and $q\geq0$, the alternating sum 
\begin{eqnarray}\label{alt sum of sp fib}
\sum_p(-1)^p\tr(g,H^p(\mx_y,\bigwedge^q(\mgp_{\mx_y/y}\otimes\ql)))
\end{eqnarray}
is an integer independent of a prime number $\ell$ distinct from the residual characteristic of $\my$. 
\end{prop}
\begin{proof}

We write $D$ and $Z$ as the sums of their irreducible components: 
$D=\sum_{i\in I}D_i$ and $Z=\sum_{j\in J}Z_j$. 
For a subset $I'\subset I$, we put $D_{I'}=\bigcap_{i\in I'}D_i$ and $\mathring D_{I'}=D_{I'}\setminus\bigcup_{i\in I\setminus I'}D_{I'\cup\{i\}}$. 
Let $I_f$ be the subset of $I$ consisting of elements $i\in I$ such that $f(D_i)\subset Z$. 
Here, since $f$ is assumed to be flat, 
we can consider the pullback $f^*Z$ of the divisor $Z$ and a map $\varphi:I_f\to J$ sending $i$ to $j$ such that $f(D_i)\subset Z_j$. 
We write the divisor $f^*Z$ 
as $\sum_{i\in I}m_iD_i$ for some $m_i\in\zz_{\geq0}$. 

Let $I'$ be a $G$-stable subset of $I$. 
Then, by Lemma \ref{log of snc}, we have a $G$-equivariant commutative diagram: 
$$
\begin{xy}
\xymatrix{
(f^{-1}\mgp_{y})|_{(\mathring D_{I'})_y}\ar[r]&
\mgp_{\mx_y}|_{(\mathring D_{I'})_y}\\
\zz^{J}\ar[r]_{\varphi_{I'}}\ar[u]_{\cong}&
\zz^{I'}\ar[u]_{\cong},
}
\end{xy}
$$
where $\varphi_{I'}$ is the homomorphism sending $e_j$ to $\sum_{i\in I' \cap \varphi^{-1}(j)}m_ie_i$. 
Thus, we get an isomorphism of $G$-sheaves
\begin{eqnarray}\label{relative log str}
(\mgp_{\mx_y/y})|_{(\mathring D_{I'})_y}\cong M_{I'},
\end{eqnarray}
where $M_{I'}$ is the constant sheaf given by the cokernel of the map $\varphi_{I'}$ 
with the $G$-sheaf structure defined by the action of $G$ on $I'$. 

We may and do assume that $G$ is a cyclic group generated by $g$. 
We consider the power set $\mathcal P(I)$ of $I$, that is, the set of all subsets of $I$, and the natural action of $G$ on $\mathcal P(I)$. 
For $A\in G\backslash\mathcal P(I)$, we set 
$\mathring D_A=\bigcup_{I'\in A}\mathring D_{I'}$. 
Then $(\mathring D_A)_y$ give a stratification 
$\mathcal X_y=\coprod_{A\in G\backslash\mathcal P(I)}(\mathring D_A)_y$ by $G$-stable locally closed subsets, and the alternating sum (\ref{alt sum of sp fib}) is equal to 
\begin{align*}
\sum_{A\in G\backslash\mathcal P(I)}\sum_p(-1)^p\tr(g,H^p_c((\mathring D_A)_y,\bigwedge^q(\mgp_{\mx_y/y}\otimes\ql)).
\end{align*}
Here, 
$\tr(g,H^p_c((\mathring D_A)_y,-))$
is nonzero only if $A$ is the orbit of a $G$-stable subset of $I$. 
Thus, it suffices to show that 
for a $G$-stable subset $I'$ of $I$, the alternating sum 
\begin{align}\label{alternating sum}
\sum_p(-1)^p\tr(g,H^p_c((\mathring D_{I'})_y,\bigwedge^q(\mgp_{\mx_y/y}\otimes\ql)).
\end{align}
is an integer independent of $\ell$ distinct from the residual characteristic of $\my$. 
By the above description (\ref{relative log str}) of $\mgp_{\mx_y/y}$, 
we obtain a $G$-equivariant isomorphism
$$
H^p_c((\mathring D_{I'})_y,\bigwedge^q(\mgp_{\mx_y/y}\otimes\ql))
\cong
H^p_c((\mathring D_{I'})_y,\mathbb Q_\ell)\otimes\bigwedge^qM_{I'}.
$$
Hence, the alternating sum (\ref{alternating sum}) is equal to 
\begin{align*}
\tr(g,\bigwedge^qM_{I'}\otimes\qq)\cdot\sum_{p}(-1)^p
\tr(g,H^p_c((\mathring D_{I'})_y,\mathbb Q_\ell)).
\end{align*}
Thus, it suffices to show that 
for a scheme $X$ separated of finite type over a separably closed field $K$ 
with an action of a finite group $G$, 
the alternating sum $\sum_q(-1)^q\tr(g,H_c^q(X,\ql))$ for $g\in G$ is an integer independent of $\ell$ different from the characteristic of $K$. 
Note that this is nothing but Theorem \ref{ellindepmonod} in the case where $K$ is separably closed, $\spec K=S$, 
and the $G$-action on $K$ is trivial. 
As we did in the beginning of the proof, we can reduce, using alterations, the proof to the case where $X$ is smooth and proper over $K$. But, in this case, the assertion follows from the Lefschetz trace formula. 
\end{proof}

\begin{rem}
We make comments on $G$-strictness of divisors with simple normal crossings. 
\begin{enumerate}
\item
We do not use $G$-strictness in the proof of Theorem \ref{ellindepmonod}: 
In the proof of Theorem \ref{ellindepmonod}, we reduce the problem to the case where $X\to\spec K$ admits a $G$-integral model which is a strictly $G$-split nodal fibration, and 
in the definition (Definition \ref{datum}, \ref{fib}) of a strictly $G$-split nodal fibration, 
we impose the condition that $Z_i$ is a divisor with $G$-strict normal crossings. 
But, for the argument after the reduction to work, it suffices to reduce to the case where $X\to\spec K$ admits a $G$-integral model which is a strictly split nodal fibration. 
\item
The $G$-strictness assumption makes the argument simpler in the sense that we can avoid to see combinatorial action of $G$ on the set $I$ of irreducible components of the divisor $D$, which appears in the proof of Proposition \ref{ellindep geom}. 
In fact, the $G$-strictness assures that the intersection $D_{I'}$ is nonempty only if the action of $G$ on $I'$ is trivial. 
Further, since we have $\sum_q(-1)^q\tr(g,\bigwedge^qM_{I'})=0$ for such $I'$ if $I'$ is nonempty, 
we get a simpler description of the alternating sum (\ref{alt sum}): it is equal to 
$\sum_p(-1)^p\tr(g,H^p((\mathring D_{\emptyset})_y,\ql)).$ 
\end{enumerate}\end{rem}

\end{document}


\title{Appendix:\\$\ell$-independence over Henselian valuation fields}
\author{Qing Lu\thanks{School of Mathematical Sciences, Beijing Normal
University, Beijing 100875, China; email: \texttt{qlu@bnu.edu.cn}.}\and
Weizhe Zheng\thanks{Morningside Center of Mathematics and Hua Loo-Keng Key
Laboratory of Mathematics, Academy of Mathematics and Systems Science,
Chinese Academy of Sciences, Beijing 100190, China; University of the
Chinese Academy of Sciences, Beijing 100049, China; email:
\texttt{wzheng@math.ac.cn}. Partially supported by National Natural Science
Foundation of China Grants 11688101, 11822110.}}
\date{}
\maketitle

\appendix\stepcounter{section}
In this appendix we prove results on $\ell$-independence and integrality of
$\ell$-adic cohomology over Henselian valuation fields with not necessarily
discrete valuations, for the action of inertia and Weil subgroups of the
Galois group. We deduce these results from relative results on compatible
systems and integral sheaves over discrete valuation fields (\cite{Zint},
\cite{Zind}) using the valuative criteria of \cite{LZ}. Our result for
inertia action (Theorem \ref{t.1}) slightly generalizes Hiroki Kato's
Theorem 6.1 and our method is different from his. We thank him for allowing
us to present our results in this appendix.

Let $K$ be a field equipped with an action of a finite group $G$. Let $X$ be
an algebraic space of dimension $d$ of finite presentation over $K$,
equipped with an equivariant action of $G$. Let $K_0=K^G$, $H=\Gal(K/K_0)$
and let $\bar K$ denote a separable closure of $K$. Let $X_{\bar
K}=X\otimes_K \bar K$. For every prime number $\ell$ invertible in $K$ and
every integer $i$, $G\times_H \Gal(\bar K/K_0)$ acts continuously on
$H^i_c(X_{\bar K},\Q_\ell)$ and $H^i(X_{\bar K},\Q_\ell)$.

\begin{thm}\label{t.1}
Let $\LL$ be a set of prime numbers not containing the characteristic
of~$K$. Then for every $(g,\sigma)\in G\times_H E_{K_0/\Spec(\Z[\LL^{-1}])}$
and every $i$, the eigenvalues of $(g,\sigma)$ acting on $H^i_c(X_{\bar
K},\Q_\ell)$ and $H^i(X_{\bar K},\Q_\ell)$ are roots of unity for $\ell\in
\LL$, and
\[
\sum_i (-1)^i\tr((g,\sigma),
H^i_c(X_{\bar K},\Q_\ell))=\sum_i (-1)^i\tr((g,\sigma), H^i(X_{\bar K},\Q_\ell))
\]
is a rational integer independent of $\ell\in \LL$.
\end{thm}
Recall from Definition 2.1 that $E_{K_0/\Spec(\Z[\LL^{-1}])}\subseteq
\Gal(\bar K/K_0)$ denotes the ramified part, namely the closure of the union
of the images of $\Gal(\bar L/L)\to \Gal(\bar K/K_0)$ for all commutative
diagrams
\[\xymatrix{\bar L & \ar[l]\bar K\\
L\ar[u] & K_0\ar[l]\ar[u]\\
\cO_L\ar[u] & \Z[\LL^{-1}]\ar[l]\ar[u]}
\]
with $\cO_L$ a strictly Henselian valuation ring of fraction field $L$, and
$\bar L$ a separable closure of $L$. In particular, if $K_0$ is the fraction
field of a Henselian valuation ring of residue characteristic $\not\in \LL$,
then $E_{K_0/\Spec(\Z[\LL^{-1}])}$ is the inertia subgroup of $\Gal(\bar
K/K_0)$.

\begin{thm}\label{t.2}
Assume that $K_0$ is the fraction field of a Henselian valuation ring $\cO$
of finite residue field $\F_q$. Then for all $(g,\sigma)\in G\times_{H}
W(\bar K/K_0)$ with $\sigma$ of degree $\nu\ge 0$, $i\in \Z$, $\ell\nmid q$,
and all eigenvalues $\alpha$ of $(g,\sigma)$ acting on $H^i_c(X_{\bar
K},\Q_\ell)$ and $H^i(X_{\bar K},\Q_\ell)$, the numbers $\alpha$, $q^{
\nu(d-i)}\alpha$, $q^{\nu d}\alpha^{-1}$, $q^{\nu i}\alpha^{-1}$ are
algebraic integers, and
\[\sum_i (-1)^i\tr((g,\sigma), H^i_c(X_{\bar K},\Q_\ell)),\quad \sum_i (-1)^i\tr((g,\sigma), H^i(X_{\bar K},\Q_\ell))\]
are rational integers independent of $\ell\neq q$. Here $W(\bar K/K_0)$
denotes the inverse image of the Weil group $W(\Fqb/\F_q)$ under the
reduction map $r\colon \Gal(\bar K/K_0)\to \Gal(\Fqb/\F_q)$, and that
$\sigma$ has degree $\nu$ means $r(\sigma)=\Fr_q^\nu$, where $\Fr_q$ denotes
the geometric Frobenius $a\mapsto a^{1/q}$.
\end{thm}

Theorem \ref{t.2} was previously known under the additional assumption that
$\cO$ is a finite field or a discrete valuation ring (\cite{Gabber},
\cite{Ochiai}, \cite[Appendix]{DE}, \cite{Zint}, \cite{Zind}).

\begin{proof}[Proof of the theorems]
In Theorem~\ref{t.1}, we may assume by continuity (see Lemma \ref{l} below)
that $K_0=L$ is the fraction field of a strictly Henselian valuation ring
$\cO$ over $\Z[\LL^{-1}]$. In this case, the equality follows from the
equivariant form of a theorem of Laumon \cite[Theorem 2.2]{IZ}. We will
prove the rest of Theorem \ref{t.1} and Theorem \ref{t.2} in parallel. In
Theorem \ref{t.2} we take $\LL$ to be the set of primes $\ell\nmid q$.

Let us prove that the alternating sums in the theorems are in $\Q$ and
independent of $\ell\in \LL$. The first step is to take care of the finite
group action. Up to replacing $G$ by $\langle g\rangle$, we may assume that
$G$ is abelian. We fix $\sigma$, of image $\bar \sigma\in H$.  Let
$t_{g,\ell}^c$ and $t_{g,\ell}$ denote the alternating sums in the theorems.
Let $\bar \Q$ be an algebraic closure of $\Q$ and let $I$ be the set of
pairs $(\ell,\iota)$, where $\ell\in \LL$ and $\iota\colon \bar\Q\to \Qlb$
is an embedding. It suffices to show the existence of $t^{(c)}_{g}\in \bar
\Q$ such that $\iota(t^{(c)}_{g})=t^{(c)}_{g,\ell}$ for all $(\ell,\iota)\in
I$ (cf.\ \cite[Remarque 1.18 (iv)]{Zind}). For any character $\chi\colon
G\to (\bar \Q)^\times$, let $V_{\iota\chi}$ denote the one-dimensional
$\Qlb$-representation of character $\iota\chi$ and let $\cL_{\iota\chi}$
denote the corresponding lisse $\Qlb$-sheaf on the quotient stack $[X/G]$.
Let $[X/G]_{\bar K}\colonequals [X/G]\otimes_{K_0} \bar K$ and let
$\pi\colon G\to H$. Then $H^*_{(c)}([X/G]_{\bar K},\cL_{\iota\chi})\simeq
(H^*_{(c)}(X_{\bar K},\Qlb)\otimes_{\Qlb} V_{\iota\chi})^{\Ker(\pi)}$, so
that
\[\tr(\sigma,H^*_{(c)}([X/G]_{\bar K},\cL_{\iota\chi})) =\frac{1}{\#\Ker(\pi)}\sum_{g\in \pi^{-1}(\bar \sigma)}t^{(c)}_{g,\ell}\iota(\chi(g)).
\]
Since every function $\pi^{-1}(\bar \sigma)\to \bar \Q$, in particular every
indicator function, is a $\bar \Q$-linear combination of the
$\chi|_{\pi^{-1}(\bar \sigma)}$'s, it suffices to show that
$(\tr(\sigma,H^*_{(c)}([X/G]_{\bar K},\cL_{\iota\chi})))$ is $I$-compatible.

The rest of the proof is similar to the proof of \cite[Theorem 1.4]{LZ},
except that the base here may have mixed characteristics. By standard limit
arguments, there exists a finitely generated sub-algebra $R\subseteq K_0$
over $\Z$ such that $X$ and the action of $G$ are defined over $B=\Spec(R)$:
there exists an algebraic space $\cX$ of finite presentation over $B$,
equipped with an action of $G$ by $B$-automorphisms, such that we have a
$G$-equivariant isomorphism $X\simeq \cX\times_B \Spec(K)$ over $K$. Let
$f\colon [\cX/G]\to B$. If the residue field of $\cO$ has characteristic
$>0$, let $p$ be its characteristic. Otherwise (this case may only occur in
Theorem 1), noting that it suffices to show $J$-compatibility for every
finite subset $J\subseteq I$, we may assume that $\LL$ is not the set of all
primes and we choose $p\not\in \LL$. In both cases, we obtain a commutative
square
\[\xymatrix{\Spec(K_0)\ar@{^{(}->}[d] \ar[r] & B_{(p)}\ar[d]\\
\Spec(\cO)\ar[r] & \Spec(\Z_{(p)}),}
\]
where $\Z_{(p)}$ denotes the Henselization of $\Z$ at the ideal $(p)$, which
is an excellent Henselian discrete valuation ring, and
$B_{(p)}=\Spec(R\otimes_\Z \Z_{(p)})$. Let $h\colon [\cX_{(p)}/G]\to
B_{(p)}$ denote the base change of $f$ to $B_{(p)}$. We still denote by
$\cL_{\iota\chi}$ the lisse $\Qlb$-sheaf on $[\cX_{(p)}/G]$ given by
$V_{\iota\chi}$. By \cite{Zind} (as summarized in \cite[Theorem 2.3]{LZ}),
$(R h_{!}\cL_{\iota\chi})$ and $(R h_{*}\cL_{\iota\chi})$ are $I$-compatible
systems on $B_{(p)}$. Moreover, since $\Spec(K_0)$ maps to the generic point
of $B_{(p)}$, $H^i([X/G]_{\bar K},\cL_{\iota\chi})\simeq
(R^ih_*\cL_{\iota\chi})_{\bar K}$ and similarly for $H^i_c$. Thus by the
valuative criterion for compatible systems \cite[Corollary 1.3]{LZ},
$(\tr(\sigma,H^*_{(c)}([X/G]_{\bar K},\cL_{\iota\chi})))$ is $I$-compatible.
This finishes the proof of rationality and $\ell$-independence.

It remains to show the assertions on eigenvalues. For these assertions we
may replace $(g,\sigma)$ by $(g,\sigma)^n$ for any $n\ge 1$. Thus we may
assume $G=\{1\}$. We conclude the proof of Theorem \ref{t.1} by the last
assertion of \cite[Corollary 3.10]{LZ} applied to $R^ih_!\Qlb$ and
$R^ih_*\Qlb$.

Finally, we prove the assertion on eigenvalues in Theorem \ref{t.2}. If
$K_0$ has characteristic $0$, we may remove the fiber at $p$ from $B$. Thus
we may assume that $B_{(p)}$ is above either the closed point or the generic
point of $\Spec(\Z_{(p)})$. Further shrinking $B$ if necessary, we may
assume that $B_{(p)}$ is regular of pure dimension $d'$, $\cX_{(p)}$ has
dimension $d''\colonequals d+d'$, and $R^ih_!\Qlb$ and $R^ih_*\Qlb$ are
lisse. By \cite[Proposition 6.4]{Zint}, $R^ih_{!}\Qlb$, $R^ih_{!}\Qlb(i-d)$,
$(R^ih_{!}\Qlb(i))\spcheck$, $(R^ih_{!}\Qlb(d))\spcheck$, $R^ih_{*}\Qlb$,
$(R^ih_{*}\Qlb(i))\spcheck$ are integral sheaves on $B_{(p)}$. Here
$\cF\spcheck\colonequals \cHom(\cF,\Qlb)$ for $\cF$ lisse. Moreover,
$Rh_*\Qlb\simeq R\cHom(Rh_!D_{\cX_{(p)}}\Qlb,\Qlb(d')[2d'])$, where we
adopted the usual normalization of dualizing functor $D_Y\colonequals
R\cHom(-,Ra^!\Qlb)$ for $a\colon Y\to \Spec(k)$. By \cite[Propositions 6.2,
6.4]{Zint}, $(R^ih_!D_{\cX_{(p)}}\Qlb)(-d'')$ and
$(R^ih_!D_{\cX_{(p)}}\Qlb(i+d''))\spcheck$ are integral. It follows that
$R^ih_{*}\Qlb(i-d)$ and $(R^ih_{*}\Qlb(d))\spcheck$ are integral. We
conclude by the valuative criterion for integrality \cite[Corollary
3.10]{LZ}.
\end{proof}

The following continuity lemma is extracted from the proof of \cite[Lemma
4.5]{LZ}. Let $P_\ell^{n}\simeq \Qlb^n$ denote the space of monic
polynomials in $\Qlb[T]$ of degree $n$.

\begin{lemma}\label{l}
Let $I$ be a finite set. For each $i$, let $\ell_i$ be a prime and $n_i\ge
0$ an integer. Let $E$ be a topological space equipped with continuous maps
$\rho_i\colon E\to \GL_{n_i}(\Qlib)$ and let $U\subseteq E$ be a dense
subspace. Let $S\subseteq \prod_i P_{\ell_i}^{n_i}$ be a subset. Assume that
$\rho_i(\sigma)$ is quasi-unipotent and $(\det (T\cdot
1-\rho_i(\sigma)))_i\in S$ for all $\sigma\in U$. Then the same holds for
all $\sigma\in E$.
\end{lemma}

\begin{proof}
Consider the continuous map $\lambda\colon E\to \prod_i P_{\ell_i}^{n_i}$
carrying $\sigma$ to $(\det (\rho_i(\sigma)-T\cdot 1))_i$. As in
\cite[Remark 2.11]{LZ}, let $P_\ell^{n,\qu}\subseteq P_\ell^{n}$ denote the
subset of polynomials whose roots are all roots of unity, which is discrete
and closed. By assumption, $\lambda(U)\subseteq S\cap \prod_i
P_{\ell_i}^{n_i,\qu}$. Thus $\lambda(U)=\lambda(E)$.
\end{proof}